\documentclass[12pt,a4paper]{scrarticle}
\usepackage[utf8]{inputenc}

\usepackage[paper=a4paper,left=15mm,right=15mm,top=25mm,bottom=25mm]{geometry}
\usepackage{amsmath}
\usepackage{amsthm}
\usepackage{amssymb}
\usepackage{aligned-overset}
\usepackage{enumitem}
\usepackage{tikz-cd}

\DeclareMathOperator{\Lan}{Lan}
\DeclareMathOperator{\Ran}{Ran}

\DeclareMathOperator{\im}{im}
\DeclareMathOperator{\coker}{coker}

\DeclareMathOperator{\ind}{ind}

\DeclareMathOperator{\codim}{codim}
\DeclareMathOperator{\len}{length}
\DeclareMathOperator{\soc}{soc}

\newcommand{\A}{{\mathcal{A}}}

\theoremstyle{plain}
\newtheorem{theorem}{Theorem}[section]
\newtheorem{lemma}[theorem]{Lemma}
\newtheorem{proposition}[theorem]{Proposition}
\newtheorem{corollary}[theorem]{Corollary}

\theoremstyle{definition}
\newtheorem{definition}[theorem]{Definition}
\newtheorem{remark}[theorem]{Remark}
\newtheorem{example}[theorem]{Example}

\begin{document}

\title{Quotient Modules of Finite Length and Their Relation to Fredholm Elements in Semiprime Rings}
\author{Niklas Ludwig}
\maketitle

\section{Abstract}
B. A. Barnes introduced so-called Fredholm elements in a semiprime ring whose definition is
inspired by Atkinson's theorem \cite{barnes_1969}. Here the socle of a semiprime ring generalizes 
the ideal of finite-rank operators on a Banach space.
There already exist many generalizations of Fredholm elements like the one discussed in $\cite{smyth_1982}$ and $\cite{maennleschmoeger_1999}$.
In this paper, we aim to see that the algebraic concept of the length of a module is
strongly related to that of Fredholm elements. This motivates another generalization of Fredholm elements
by requiring for an element $a\in\A$ that the $\A$-modules of the form $\A/\A a$ and $\A/a\A$ are of finite length.
We are particularly interested in sufficient conditions for our generalized Fredholm elements to be Fredholm. 
In a unital C$^*$-algebra $\A$ we shall even see that an element $a\in\A$ is Fredholm if and only if the $\A$-modules $\A/\A a$ and $\A/a\A$
both have finite length.

\section{Preliminary}
Let $\A$ be a ring and $S\subseteq\A$. Then we define the annihilators
\begin{align*}
\Lan(S)&\coloneqq\{a\in\A: as=0 \text{ for all } s\in S\} \\
\Ran(S)&\coloneqq\{a\in\A: sa=0 \text{ for all } s\in S\}.
\end{align*}
Note that $\Lan(S)\subseteq\A$ is a left ideal and $\Ran(S)\subseteq\A$ is a right ideal.
A semiprime ring $\A$ is called right annihilator ring if for every modular maximal left ideal $m\subsetneq\A$ it holds that $\Ran(m)\neq0$ (see \cite{barnes_1966}).
Now assume $\A$ is semiprime. Then, we write $\soc(\A)$ for the socle of $\A$.
Moreover, we call a left (right) ideal of $\A$ of finite order if it can be written as a finite sum of minimal left (right) ideals.
In this case, the minimal number of such minimal left (right) ideals needed is called the order of this left (right) ideal (see \cite{barnes_1968}).
If $\A$ is unital, an element $a\in\A$ is called Fredholm if the image $[a]$
of the quotient map is invertible in $\A/\soc(\A)$ (see \cite{barnes_1969}). 
Similarly, we call an element $a$ semi$^+$-Fredholm if $[a]$ is left invertible and 
semi$^-$-Fredholm if $[a]$ is right invertible in $\A/\soc(\A)$. \\
Caution is advised as we always use the term semisimple in the sense of $\A$-modules and call a ring $J$-semisimple instead
if the intersection of its modular maximal left ideals is trivial.
Throughout the paper, we shall only consider Banach algebras and C$^*$-algebras over the complex numbers.
\section{Fredholm Ideals}

In the entire section $\A$ denotes a semiprime unital ring.
For us, it will be more convenient to study Fredholm elements by using 
the theory of one-sided ideals.

\begin{definition}
We call a left ideal $L\subseteq\A$ Fredholm if it contains a Fredholm element.
\end{definition}

The following is an important result due to Barnes:

\begin{lemma} \label{soc_idmp}
Assume $F\subseteq\A$ is a left (right) ideal of finite order. Then there exists an idempotent $p\in\soc(\A)$ satisfying $\A p = F$ ($p\A=F$).
\end{lemma} 
\begin{proof}
(see \cite{barnes_1968}, Theorem 2.2)
\end{proof}

\begin{proposition} \label{ch_fred_idp}
A left ideal $L\subseteq\A$ is Fredholm if and only if there exists an idempotent $p\in\soc(\A)$ such that $L=\A(1-p)$.
\end{proposition}
\begin{proof}
Suppose there is a Fredholm element $a\in L$. Then, there exists an idempotent $q\in\soc(\A)$ such that 
$\A a = \A (1-q)$ holds (see \cite{barnes_1969}, Theorem 2.3). We set $L'\coloneqq L\cap \A q$. As $L'$ is of finite order,
there exists an idempotent $p'\in\A$ such that $L'=\A p'$ by Lemma \ref{soc_idmp}. Now, we define $p\coloneqq q-p'$. Note that $p'q=p'$ and hence $p'p=0$. \\
In order to show $L\subseteq\A(1-p)$, let $x\in L$. Then $xq=x-x(1-q)\in L$ and therefore $xq\in L'$.
Thus, there exists $b\in\A$ with $xq=bp'$. We finally see that
$$xp=x(1-q)p+xqp = x(1-p-p')p+bp'p = 0. $$
For the other inclusion, let $x\in\A$ with $xp=0$. Again we write $x = x(1-q) + xq$ and note that $x(1-q)\in L$.
As $xq = xp+xp' = xp'$, it finally follows that $L=\A(1-p)$.
\end{proof}

\begin{corollary}
Let $a\in\A$. Then $a$ is semi$^+$-Fredholm if and only if the left ideal $\A a$ is Fredholm.
\end{corollary}

\begin{proposition} \label{inter_fred}
A finite intersection of Fredholm left ideals is again Fredholm.
\end{proposition}
\begin{proof}
Suppose $L\subseteq\A$ and $L'\subseteq\A$ both are Fredholm left ideals. By Proposition \ref{ch_fred_idp} there exist $p,q\in\soc(\A)$
with $L=\A(1-p)$ and $L'=\A(1-q)$. As $p\A+q\A$ is of finite order, Lemma \ref{soc_idmp} implies that there exists $r\in\soc(\A)$ such that 
$$ p\A + q\A = r\A. $$
We show that $L\cap L' = \A(1-r)$: Indeed let $x\in L\cap L'$. As $xp=xq=0$
and $r\in p\A + q\A$, we obtain that $xr=0$. Conversely, pick $x\in\A$ with $xr=0$.
Clearly $p,q\in r\A$. This already implies $xp=xq=0$ and thus $x\in L\cap L'$.
\end{proof}

We show the following statement in order to bring the concept of length into play.
    
\begin{lemma} \label{order_len}
Let $L\subseteq\A$ be a left ideal. Then $L$ is of finite order if and only if $L$ has finite length as an $\A$-module.
In this case, the order of $L$ and $\len_\A(L)$ coincide. 
\end{lemma}
\begin{proof}
Assume $L$ is of finite order. Then, there exist minimal left ideals $m_1,\dots,m_n\subseteq L$ such that $L=\oplus_{j=1}^n m_j$ 
and $n$ is the order of $L$ (see \cite{barnes_1968}, Theorem 2.2).
Clearly, we have $\len_\A(m_j)=1$. As the length is additive, we get
$$ \len_\A(L) = \sum_{j=1}^n \len_\A(m_j) = n. $$
Conversely, suppose $n\coloneqq\len_\A(L)<\infty$. We prove the statement by induction on $n$: \\
If $n=0$, we are done. \\
If $n>0$, the left ideal $L$ contains a minimal left ideal $m\subseteq L$ since $\len_\A(L)<\infty$. Now, there exists a minimal idempotent $p\in\A$ such that $m=\A p$.
Thus, we can write $\A p\oplus [L\cap\A(1-p)]= L$: We only show $L\subseteq \A p\oplus [L\cap\A(1-p)]$: Indeed for $x\in L$,
we have $x(1-p)=x-xp\in L$ and hence $x(1-p)\in L\cap\A(1-p)$. This yields $x=xp+x(1-p)\in\A p\oplus [L\cap\A(1-p)]$.
Finally 
$$\len_\A(L\cap\A(1-p))=n-1$$
and the claim follows by using the induction hypothesis.
\end{proof}

\begin{definition}
Let $L\subseteq\A$ be a left ideal. Then we set $\varrho(L)\coloneqq \len_{\A}(\Ran(L))$ and
$\xi(L)\coloneqq \len_\A(\A/L)$. If $L=\A a$ for some $a\in\A$, we write $\varrho_l(a)\coloneqq\varrho(L)$ and
$\xi_l(a)\coloneqq\xi(L)$. The corresponding definition can be made for right ideals. In particular for $a\in\A$,
we write $\varrho_r(a)\coloneqq\len_{\A}(\Lan(a))$ and $\xi_r(a)\coloneqq\len_{\A}(\A/a\A)$.
\end{definition}

\begin{lemma}
Let $a\in\soc(\A)$. Then both $a\A$ and $\A a$ are of the same order. 
\end{lemma}
\begin{proof}
Let $\{p_1,\dots,p_n\}\subseteq \A a$ be a maximal orthogonal set of minimal idempotents.
Then $\A a = \A p$ where $p\coloneqq p_1+\dots+p_n$ (see \cite{barnes_1968}, Theorem 2.2) and $n$ coincides with the order of $\A a$.
Therefore, there exists $b\in\A$ satisfying $a=bp$ and the map $p\A\to a\A,x\mapsto bx$ is surjective. Thus, we conclude that the order of
$a\A$ is bounded by $n$. The other inequality is shown analogously.
\end{proof}

\begin{proposition} \label{ch_fred_len}
Let $L\subseteq\A$ be a left ideal. Then $L$ is Fredholm if and only if $\xi(L)\leq\varrho(L)<\infty$.
In this case, $\xi(L)=\varrho(L)$.
\end{proposition}
\begin{proof}
Suppose that $L$ is Fredholm. By Proposition \ref{ch_fred_idp}, there exists an idempotent $p\in\soc(\A)$ such that $L=\A(1-p)$.
As the one-sided ideals $p\A=\Ran(L)$ and $\A p$ are of the same order and $\A/L\cong \A p$, Lemma \ref{order_len} yields $\len_\A{\A/L}=\varrho(L)<\infty$. \\
Conversely, assume $\len_\A{\A/L}\leq\varrho(L)<\infty$. Lemma \ref{order_len} implies that $\Ran(L)$ is of finite order and thus
there exists an idempotent $p\in\soc(\A)$ satisfying $p\A = \Ran(L)$ by Lemma \ref{soc_idmp}.
We set $L'\coloneqq \A(1-p)$ and we already know that $\len_\A \A/L'=\varrho(L')=\varrho(L)$. Since $L\subseteq L'$, we can
consider the surjective map 
$$ \nu\colon \A/L\to \A/L', [x]\mapsto [x]. $$
As $\len_\A{\A/L} \leq \varrho(L) = \len_\A \A/L'<\infty$, the map $\nu$ is even an isomorphism. Hence, $L=L'$ and we are done.
\end{proof}

\begin{theorem} \label{ch_fred_th}
Let $L\subseteq\A$ be a left ideal. Then the following statements are equivalent:
\begin{enumerate}[label={(\roman*)}]
\item $L$ is Fredholm,
\item $L=\Lan(\Ran(L))$ and $\varrho(L)<\infty$,
\item there exist $n\in\mathbb{N}_0$ and idempotent elements $(p_j)_{j=1,\dots,n}\subseteq\A$ such that for $j=1,\dots,n$ the left ideals $\A p_j$ are maximal 
and $L=\bigcap_{j=1}^n \A p_j$.
\end{enumerate}
\end{theorem}
\begin{proof}
``$(i)\Rightarrow(ii)$'': This is a direct consequence of Proposition \ref{ch_fred_idp} and Lemma \ref{order_len}. \\
``$(ii)\Rightarrow(iii)$'': By Lemma \ref{order_len}, we know that $\Ran(L)$ has finite order and thus there exist minimal idempotent elements $q_1,\dots q_n\in\A$
such that $\Ran(L)=q_1\A+\dots+q_n\A$. We set $p_j\coloneqq 1-q_j$ for $j=1,\dots,n$ and note that $\A p_j$ are maximal left ideals. Finally, we can write
$$ L = \Lan(\Ran(L)) = \Lan(q_1\A+\dots+q_n\A) = \bigcap_{j=1}^n \A p_j. $$
``$(iii)\Rightarrow(i)$'': Note that $1-p_j$ are minimal idempotent elements for $j=1,\dots,n$ and hence $\A p_j$ are Fredholm left ideals by Proposition \ref{ch_fred_idp}.
The statement finally follows by Proposition \ref{inter_fred}.
\end{proof}

\begin{definition}
We call a left ideal of $\A$ semi-maximal if it can be written as an intersection of maximal left ideals.
\end{definition}

Next, we repeat some standard results of non-commutative ring theory:
\begin{proposition} \label{ch_semisimple}
Let $L\subseteq\A$ be a left ideal. Then the following assertions are equivalent:
\begin{enumerate}[label={(\roman*)}]
\item the left $\A$-module $\A/L$ is semisimple,
\item $L$ is semi-maximal and $\xi(L)<\infty$,
\item there exist $n\in\mathbb{N}_0$ and maximal ideals $m_1,\dots,m_n\subset\A$ such that $L=\bigcap_{j=1}^n m_j$.
\end{enumerate}
\end{proposition}
\begin{proof}
``$(i)\Rightarrow(ii)$'': As $\A/L$ is semisimple and $\A/L$ is finitely generated ($[1]\in\A/L$), we see that $\len_\A{\A/L}<\infty$.
Further, $\A/L$ is J-semisimple (see \cite{lam2013first}, (4.14) Theorem) and thus $L$ 
can be written as the intersection of all maximal left ideals that contain $L$. \\
``$(ii)\Rightarrow(iii)$'': WLOG we have $L\subsetneq\A$. Then $L$ is contained in a maximal left ideal $m_1\subset\A$ and we set $L_1\coloneqq m_1$.
For $j\geq 2$, we choose $m_j$ recursively and distinguish two cases: 
If $L_{j-1}\subseteq m$ for all maximal left ideals $L\subseteq m\subset\A$, we terminate. Otherwise we can pick a maximal left ideal $L\subseteq m_j\subset\A$ 
with $L_{j-1}\nsubseteq m_j$ and set $L_j\coloneqq L_{j-1} \cap m_j$. This procedure terminates after finitely many iterations as
$$ L \subseteq\dots\subsetneq L_2\subsetneq L_1$$
and $\len_\A{\A/L}<\infty$. Let $n$ be one less the number of iterations. Then, we obtain $L=L_n$ since $L$ is semi-maximal. \\
``$(iii)\Rightarrow(i)$'': We simply note that there is an injective $\A$-linear map $\A/L\to\oplus_{j=1}^n\A/m_j$.
\end{proof}

\begin{corollary}
If $L\subseteq\A$ is a Fredholm left ideal, then $\A/L$ is a semisimple $\A$-module.
\end{corollary}

The following result due to B. A. Barnes states that a unital ring where all maximal left ideals are Fredholm is semisimple.
We shall later see that the situation can be very different for principal one-sided ideals.

\begin{proposition}
The following assertions are equivalent:
\begin{enumerate}[label={(\roman*)}]
\item $\A$ is a unital right annihilator ring,
\item every left ideal is Fredholm,
\item $\A$ is semisimple,
\item every maximal left ideal is Fredholm.
\end{enumerate}
\end{proposition}
\begin{proof}
(see \cite{barnes_1966}, Theorem 4.3) and (see \cite{barnes_1969}, Theorem 2.4)
\end{proof}

\section{Weak-Fredholm Elements}

In the entire section, $\A$ denotes a unital ring.
Note that the following definition is actually weaker than the definition of Fredholm elements.

\begin{definition}
Let $a\in\A$. Then we call $a$ weak$^+$-Fredholm if $\xi_l(a)<\infty$.
Similarly, we call $a$ weak$^-$-Fredholm if $\xi_r(a)<\infty$.
Moreover, we call an element weak-Fredholm if it is both weak$^+$-Fredholm and weak$^-$-Fredholm.
\end{definition}

\begin{example}
Suppose $\A$ is a commutative unital Banach algebra. Then an element $a\in\A$ is weak-Fredholm if and only if 
the multiplication operator $M_a\colon\A\to\A,\,x\mapsto xa$ is semi$^-$-Fredholm. Further, $\xi_l(a)=\xi_r(a)=\dim(\coker M_a)$.
\end{example}
\begin{proof}
Suppose $M_a$ is semi$^-$-Fredholm. Then, it is clear that $\xi_l(a)=\xi_r(a)<\infty$. \\
Conversely, assume that $a$ is weak-Fredholm. Then, there exists a composition series
$$ \A a = I_0 \subsetneq I_1 \dots \subsetneq I_n = \A $$
such that the quotients $I_j/I_{j-1}$ are simple $\A$-modules for $j=1,\dots,n$ and $n=\xi_l(a)$. Hence, there exist maximal ideals $m_j\subset\A$
satisfying $I_j/I_{j-1}\cong\A/m_j$. Since $\A$ is a unital Banach algebra
$$ I_j/I_{j-1}\cong\A/m_j\cong\mathbb{C}$$
for $j=1,\dots,n$ and thus $n=\len_\mathbb{C}(\A/\A a)=\dim(\coker M_a)$.
\end{proof}

\begin{proposition}
Suppose $a,b\in\A$ are weak$^+$-Fredholm. Then $ab$ is again weak$^+$-Fredholm and $\xi_l(ab)\leq\xi_l(a)+\xi_l(b)$.
\end{proposition}
\begin{proof}
Consider the map $\phi\colon\A\to\A/\A ab\,,x\mapsto [xb]$.
As $\A a\subseteq\ker\phi$, there exists an $\A$-linear map $\overline{\phi}\colon\A/\A a\to\A/\A ab$ 
with $\im\overline{\phi} = \A b/\A ab$. Thus, the following sequence is exact:
$$ \begin{tikzcd} \A/\A a \arrow[r, "\overline{\phi}"] & \A/\A ab \arrow[r]  & \A/\A b \arrow[r]  & 0 \end{tikzcd} $$
As the length is additive in short exact sequences, the claim follows.
\end{proof}

In general, it is not always true that a weak$^+$-Fredholm element $a\in\A$ satisfies $\varrho_l(a)<\infty$. 
In the following, we want to find sufficient conditions that ensure $\varrho(L)<\infty$ for a left ideal $L\subseteq\A$.

\begin{lemma} \label{lemma_semi_min}
Let $L\subsetneq\A$ be a left ideal such that $\A/L$ is a semisimple $\A$-module. Then, there exist $q\in\A\setminus L$ such that 
$L+\A(1-q)$ is a maximal left ideal and $Lq\subseteq L$.
\end{lemma}
\begin{proof}
$\A/L$ contains a minimal submodule $N\subseteq\A/L$. As $\A/L$ is semisimple there exists a submodule $M\subseteq\A/L$ such that $N\oplus M = \A/L$.
Thus, we can write $[1]=[q]+[m]$ where $[q]\in N$ and $[m]\in M$. Clearly, $[q]\neq0$ and hence $q\in\A\setminus L$.
For $x\in L$ we have $[0]=[x]=[xq]+[xm]$. Since the sum is direct, this yields $[xq]=0$ and thus $xq\in L$. This shows $Lq\subseteq L$.
In order to show that $L+\A(1-q)$ is a maximal left ideal, we prove
$$ M = \{ a[1-q]: a\in\A \}. $$
Evidently, $[1-q]=[m]\in M$. Conversely, let $[x]\in M$. Then $[xq]=[x]-[xm]\in M$ and therefore $[xq]=0$. Thus $[x]=[xm]=x[1-q]$.
\end{proof}

\begin{theorem}
Let $\A$ be a J-semisimple unital ring and $L\subseteq\A$ be a left ideal such that $\A/L$ is a semisimple $\A$-module. Then $\varrho(L)<\infty$.
\end{theorem}
\begin{proof}
We by proof the statement by induction on $n\coloneqq\len_\A\A/L$. The case $n=0$ is trivial. \\
If $n=1$, then $L$ is a maximal left ideal. Let $0\neq R\subseteq\Ran(L)$ be a right ideal and choose $0\neq x\in R$.
As $\A$ is J-semisimple, there exists $b\in\A$ such that $1-xb\notin \A^\times$. Clearly, we have
$$ L\subseteq \A(1-xb) \subsetneq \A. $$
By the maximality of $L$, we get $L=\A(1-xb)$ and thus $1-xb\in L$. This yields $R=\Ran(L)$ and we obtain $\len(\Ran(L))\leq 1$. \\ 
Now, assume $n\geq 2$. Lemma \ref{lemma_semi_min} yields the existence of an element $q\in\A$ as described above. In order to get
$$ \Ran(L) = [\Ran(L)\cap(1-q)\A]+[\Ran(L)\cap q\A] = \Ran(L+\A q) + \Ran(L+\A(1-q)), $$
we only need to show $\Ran(L)\subseteq[\Ran(L)\cap(1-q)\A]+[\Ran(L)\cap q\A]$. Indeed, let $x\in\Ran(L)$. We can write
$ x = (1-q)x+qx $. Since $aq\in L$ for all $a\in L$, we obtain $(1-q)x,qx\in\Ran(L)$ and thus the inclusion holds.
Now, apply the induction hypothesis.
\end{proof}

\begin{definition}
We say a semiprime ring $\A$ has essential socle if $I\cap\soc(\A)=0$ already implies $I=0$ for all ideals $I\subseteq\A$.
\end{definition}

Let us recall some basic characterizations of semiprime rings with essential socle:

\begin{proposition} \label{ch_ess_soc}
Suppose $\A$ is semiprime. Then, the following statements are equivalent:
\begin{enumerate}[label={(\roman*)}]
\item the ring $\A$ has essential socle,
\item for every $a\in\A$ with $a\soc(\A)=0$, it already holds that $a=0$,
\item for every $a\in\A$ with $\soc(\A)a=0$, it already holds that $a=0$.
\end{enumerate}
\end{proposition}
\begin{proof}
``$(i)\Rightarrow(iii)$'': Assume $\soc(\A)a=0$ and let $I\subseteq\A$ be the ideal generated by $a$.
We want to show that $I\cap\soc(\A)=0$: Indeed, let $x\in I\cap\soc(\A)$. 
Lemma \ref{soc_idmp} yields that there exists an idempotent $p\in\soc(\A)$ such that $\A x = \A p$.
Thus $p\in I$ and there exist elements $x_1,\dots, x_n\in\A$ and $y_1,\dots, y_n\in\A$ such that
$$ p = p^2 = p \sum_{j=1}^n x_j a y_j = \sum_{j=1}^n (px_j) a y_j = 0 $$
as $px_j\in\soc(\A)$ for $j=1,\dots,n$. We hence see that $I\cap\soc(\A)=0$ and as $\A$ has essential socle we conclude $a=0$. \\
``$(iii)\Rightarrow(i)$'': Let $I\subseteq\A$ be an ideal with $I\cap\soc(\A)=0$. Then we pick an arbitrary $a\in I$ and note that 
$\soc(\A)a\subseteq I\cap\soc(\A)=0$. Then, we already get $a=0$ and thus $I=0$. \\
``$(i)\iff(ii)$'': This is done analogously.
\end{proof}

\begin{lemma} \label{soc_bound_len}
Assume $\A$ has essential socle and let $R\subseteq\A$ be a right ideal. If there exists an upper bound $M\geq0$
such that for all idempotent elements $p\in\soc(\A)\cap R$ we have $\len_\A(p\A)\leq M$, then $\len_\A(R)\leq M$.
\end{lemma}
\begin{proof}
Choose an idempotent $p\in\soc(\A)\cap R$ such that $\len_\A(p\A)$ is maximal and suppose $p\A\subsetneq R$.
Then, there exists $r\in R\setminus p\A$. We write $r = pr+(1-p)r$ and deduce that $0\neq (1-p)r\in R$.
Proposition \ref{ch_ess_soc} yields that there exists $0\neq x\in(1-p)r\soc(\A)\subseteq R$. By Lemma \ref{soc_idmp}, there exists an idempotent 
$q\in\soc(\A)$ such that $q \A = x \A + p \A$. As $q\in R$
and $p \A\subseteq q\A$, we obtain $p\A=q\A$ by the maximality of $\len_\A(p\A)$. But then $x \in p\A\cap(1-p)\A=0$ 
and we reach a contradiction!
\end{proof}

\begin{theorem}
Assume $\A$ has essential socle. If $L\subseteq\A$ is a left ideal such that $\xi(L)<\infty$, then $\varrho(L)\leq\xi(L)<\infty$.
\end{theorem}
\begin{proof}
Let $p\in\soc(\A)\cap\Ran(L)$ be an idempotent. We show that $\len_\A(p\A)\leq\xi(L)$:
Indeed, we have $L\subseteq \A(1-p)$. We then get by using Proposition \ref{ch_fred_len}
$$ \len_\A(p\A) = \varrho_l(1-p) = \xi_l(1-p) \leq \xi(L). $$
The claim finally follows by Lemma \ref{soc_bound_len}.
\end{proof}

\begin{corollary}
Suppose $\A$ has essential socle. If $a\in\A$ is weak$^+$-Fredholm, then $\varrho_l(a)<\infty$.
\end{corollary}

From now on, we assume that $\A$ satisfies the following two properties:
\begin{enumerate}[label={(\roman*)}]
\item for all weak$^+$-Fredholm elements $a\in\A$ it holds that $\varrho_l(a)<\infty$, 
\item for all weak$^-$-Fredholm elements $a\in\A$ is holds that $\varrho_r(a)<\infty$.
\end{enumerate}
In this situation, the following definition results in finite numbers:
\begin{definition}
Let $a\in\A$. If $a$ is weak$^+$-Fredholm, we set $\zeta_l(a)\coloneqq\xi_l(a)-\varrho_l(a)$ and analogously
if $a$ is weak$^-$-Fredholm, we set $\zeta_r(a)\coloneqq\xi_r(a)-\varrho_r(a)$.
\end{definition}

\begin{proposition}
Suppose $\A$ is semiprime and let $a\in\A$. If $a$ is weak$^+$-Fredholm, then $\zeta_l(a)\geq0$.
Further, $\zeta_l(a)=0$ if and only if $a$ is semi$^+$-Fredholm. A corresponding result holds for weak$^-$-Fredholm elements.
\end{proposition}
\begin{proof}
Suppose $a$ is weak$^+$-Fredholm. As $\varrho_l(a) < \infty$, there exists an idempotent $p\in\A$
such that $p\A=\Ran(a)$ by Lemma \ref{soc_idmp}. Since $\A a \subseteq \A (1-p)$ we obtain
$$ \varrho_l(a) = \varrho_l(1-p) = \xi_l(1-p) \leq \xi_l(a). $$
The remaining part follows by Proposition \ref{ch_fred_len} and Proposition \ref{ch_fred_idp}.
\end{proof}

In order to simplify proofs, we need a notion of Fredholm maps between modules:

\begin{definition}
Let $M_1,M_2$ be left $\A$-modules and $\varphi\colon M_1\to M_2$ be an $\A$-linear map. Then we call $\varphi$ Fredholm if $\len_\A(\ker\varphi)<\infty$
and $\len_\A(\coker\varphi)<\infty$. In this case, we set 
$$\ind(\varphi)\coloneqq \len_\A(\ker\varphi)-\len_\A(\coker\varphi). $$
\end{definition}

The following lemma is just a generalized version of the usual additivity of the index in the vector space case:

\begin{lemma} \label{atkinson_fred}
Let $M_1,M_2,M_3$ be left $\A$-modules and $\varphi\colon M_1\to M_2$ as well as $\psi\colon M_2\to M_3$ be $\A$-linear maps
that are Fredholm. Then 
$$\ind(\psi\circ\varphi)=\ind(\psi)+\ind(\varphi).$$
\end{lemma}
\begin{proof}
(see \cite{xiong_2020}, Theorem 4)
\end{proof}

\begin{theorem}
Suppose $a,b\in\A$ are weak-Fredholm elements. Then
$$ \zeta_l(ab)+\zeta_r(ab) = \zeta_l(a)+\zeta_r(a) + \zeta_l(b)+\zeta_r(b). $$
\end{theorem}
\begin{proof}
Consider the maps $\varphi\colon\A\to\A\,,x\mapsto xa$ and $\psi\colon\A\to\A\,,x\mapsto xb$.
Then $\varphi$ and $\psi$ are Fredholm maps and Lemma \ref{atkinson_fred} yields that:
\begin{align*}
\varrho_r(ab)-\xi_l(ab) = \ind(\psi\circ\varphi) = \ind(\psi)+\ind(\varphi) = \varrho_r(a)-\xi_l(a) + \varrho_r(b)-\xi_l(b)
\end{align*}
Similarly, we obtain $\varrho_l(ab)-\xi_r(ab) = \varrho_l(a)-\xi_r(a) + \varrho_l(b)-\xi_r(b)$.
By adding up both equations, the claim follows.
\end{proof}

\begin{corollary}
Assume $\A$ is semiprime and let $a,b\in\A$ be weak-Fredholm elements such that $\zeta_l(ab),\zeta_r(ab)\leq1$. Then $a$ is Fredholm or $b$ 
is Fredholm or both $a$ and $b$ are semi-Fredholm.
\end{corollary}

Note that by using the next result, one can boil up an alternative approach to the first equivalence of Corollary \ref{equiv_weak_semi}.

\begin{corollary}
Let $a\in\A$ be a weak-Fredholm element and $n\in\mathbb{N}$ such that there exists a n-th root of $a$ (i.e. $s\in\A$ satisfying $s^n=a$). Then 
$$n \mid (\zeta_l(a)+\zeta_r(a)).$$
In particular, if $\A$ is semiprime and $n$ can be chosen arbitrarily large, then $a$ is Fredholm.
\end{corollary}

\section{Fredholm theory in $C^*$-algebras}

In the entire section $\A$ denotes a unital $C^*$-algebra.

\begin{remark}
We note that the semi-maximal left ideals of $\A$ are precisely the closed left ideals (see \cite{murphy2014c}, 5.3.3. Theorem).
\end{remark}

Next, we repeat a well-known result of the theory of $C^*$-algebras:

\begin{lemma} \label{finite_proj}
Let $L\subseteq\A$ be a closed left ideal that is finitely generated. Then, $L$ is generated by a projection.
\end{lemma}
\begin{proof}
(see \cite{blecher_2014}, Lemma 2.1)
\end{proof}

\begin{theorem} \label{fin_gen_proj} 
Let $L\subseteq\A$ be a finitely generated left ideal such that the $\A$-module $\A/L$ is Noetherian. Then, $L$ is Fredholm.
\end{theorem}
\begin{proof}
We prove that every left ideal $L\subseteq I\subseteq\A$ is generated by an idempotent.
Clearly, we have that $\overline{I}$ is finitely generated and thus Lemma \ref{finite_proj} implies
that there exists an idempotent $p\in\A$ satisfying $\overline{I}=\A p$.
Now, we can see that $I=\A p$: Indeed, there exists $x\in I$ with $\|x-p\|<1$.
By the Neumann series, we obtain $yx=p$ for some $y\in\A$ and hence $I=\A p$.
In particular, we see that $L$ is generated by an idempotent and that $\A/L$ is semisimple.
The statement finally follows by Theorem \ref{ch_fred_th}.

\end{proof}

Note that the following corollary is not so relevant for the remaining paper
but still shows how the preceding theorem can be applied to the theory of $\A$-modules.
\begin{corollary} 
Let $M$ be a Noetherian left $\A$-module of finite presentation. If $M$ is generated by $n$ elements,
there exist left ideals of finite order $L_1,\dots, L_n\subseteq\A$ such that
$$ M \cong L_1 \oplus \dots \oplus L_n.$$
\end{corollary}
\begin{proof}
We show that statement by induction over $n$. If $n=0$, the statement is trivial.
Now, assume $n>0$. Then, there exists a surjective $\A$-linear map $\varphi\colon\A^n\to M$ such that $E\coloneqq\ker\varphi$ is finitely generated.
Let $\pi\colon\A^n\to\A,\,(x_1,\dots,x_n)\to x_n$ and note that by the homomorphism theorem, there exists a surjective map $\A^n/E\to\A/\pi(E)$.
By Theorem \ref{fin_gen_proj}, we have that $\pi(E)$ is Fredholm.
Now, there exists an idempotent $p\in\soc(\A)$ such that $\pi(E)=\A(1-p)$ by Proposition \ref{ch_fred_idp}.
We set $L_n\coloneqq \A p \cong \A/\pi(E)$ and $F\coloneqq \A^{n-1}\times\{0\}$. We consider the short exact sequence
$$ \begin{tikzcd} 0 \arrow[r] & E\cap F \arrow[r] & E \arrow[r,"\pi|_E"] & \pi(E) \arrow[r]  & 0 \end{tikzcd} $$
and note that it splits as $\pi(E)$ is projective. Hence, $E\cap F$ is finitely generated and the left $\A$-module $M'\coloneqq F/(E\cap F)$
is of finite presentation. Further, the short exact sequence
$$ \begin{tikzcd} 0 \arrow[r] & M' \arrow[r] & M \arrow[r] & L_n \arrow[r]  & 0 \end{tikzcd} $$
splits as $L_n$ is projective. Therefore, $M\cong M'\oplus L_n$. As $M'$ is a Noetherian $\A$-module of finite presentation
that is generated by $n-1$ elements, the induction hypothesis yields the claim.
\end{proof}

We can also use the preceding theorem to see that in the case of $C^*$-algebras 
every weak-Fredholm element is Fredholm.
\begin{corollary} \label{equiv_weak_semi}
Let $a\in\A$. Then, the following assertions are equivalent:
\begin{enumerate}[label={(\roman*)}]
\item $a$ is weak$^+$-Fredholm,
\item $a$ is semi$^+$-Fredholm,
\item the $\A$-module $\A / \A a$ is semisimple,
\item $\varrho_l(a)<\infty$ and $\A a$ is closed.
\end{enumerate}
\end{corollary}

\begin{definition}
Let $L\subseteq\A$ be a left ideal. Then we set $\delta(L)\coloneqq\codim(L+L^*)$.
\end{definition}

\begin{proposition}
Let $L_1,L_2\subseteq\A$ be closed left ideals such that $L_1\subseteq L_2$. Then $\delta(L_2)\leq\delta(L_1)$.
Moreover, if $\delta(L_1)=\delta(L_2)<\infty$ we have $L_1=L_2$.
\end{proposition}
\begin{proof}
It is clear that $\delta(L_2)\leq\delta(L_1)$. Now suppose $\delta(L_1)=\delta(L_2)<\infty$
and let $\varphi\colon\A\to\mathbb{C}$ be a positive linear functional that vanishes on $L_1$.
We show that $\varphi(L_2)=0$: As $\delta(L_1)=\delta(L_2)<\infty$, we have $L_1+L_1^*=L_2+L_2^*$. Then, we get
$$\varphi(L_2)\subseteq\varphi(L_2+L_2^*)=\varphi(L_1+L_1^*) = \{0\}$$
since $\varphi$ is self-adjoint (see \cite{murphy2014c}, 3.3.2. Theorem).
As $\varphi$ is an arbitrary positive functional on $\A$ that vanishes von $L_1$, the claim follows (see \cite{murphy2014c}, 5.3.2. Theorem).
\end{proof}

\begin{proposition}
Let $L\subseteq\A$ be a closed left ideal. If $\delta(L)<\infty$, then $\A/L$ is a semisimple $\A$-module.
\end{proposition}
\begin{proof}
By using Proposition \ref{ch_semisimple}, we only need to show that $\xi(L)<\infty$. Now, let 
$$ L \subsetneq L_1 \subsetneq \dots \subsetneq L_n = \A $$
be a chain of left ideals. We then see that $n\leq\delta(L)$ as $ \delta(L) > \delta(L_1) > \dots > \delta(L_n)$.
\end{proof}

\begin{theorem}
Let $a\in\A$. Then, the following statements are equivalent:
\begin{enumerate}[label={(\roman*)}]
\item $a$ is weak$^+$-Fredholm,
\item $a$ is semi$^+$-Fredholm,
\item $\delta(\A a)<\infty$ and $\A a$ is closed.
\end{enumerate}
\end{theorem}
\begin{proof}
We only need to show ``$(ii)\Rightarrow(iii)$'':
By Proposition \ref{ch_fred_idp}, there exists an idempotent $p\in\soc(\A)$ such that $\A a = \A (1-p)$.
In $C^*$-algebras, there even exists a projection $q\in\soc(\A)$ such that $\A a = \A (1-q)$ (see \cite{davidson1996c}, Proposition IV.1.1).
We finally obtain a decomposition
$$ \A = \A a + a^*\A + q\A q $$
where $\dim q\A q < \infty$ (see \cite{alexander_1968}, Theorem 7.2).
\end{proof}

\bibliographystyle{plain}
\bibliography{references} 

\end{document}